\documentclass[11pt]{amsart}

\date{May 27, 2007}
\usepackage{amsxtra}

\newtheorem{theorem}{Theorem}[section]
\newtheorem{proposition}[theorem]{Proposition}
\newtheorem{lemma}[theorem]{Lemma}
\newtheorem{corollary}[theorem]{Corollary}

\theoremstyle{definition}

\theoremstyle{remark}

\newtheorem{remark}[theorem]{Remark}

\numberwithin{equation}{section}

\renewcommand{\epsilon}{\varepsilon}

\newcommand{\imunit}{{\rm i}}

\newcommand{\N}{\mathbb{N}}

\renewcommand{\phi}{\varphi}

\newcommand{\R}{\mathbb{R}}

\newcommand{\Sph}{\mathbb{S}}

\DeclareMathOperator{\curl}{curl}

\DeclareMathOperator{\im}{Im}
\DeclareMathOperator{\re}{Re}
\DeclareMathOperator{\supp}{supp}
\DeclareMathOperator{\sgn}{sgn}

\begin{document}

\title[Eigenvalue estimates]{Eigenvalue estimates for magnetic Schr\"odinger operators in domains}

\begin{abstract}
  Inequalities are derived for sums and quotients of eigenvalues of
  magnetic Schr\"o\-din\-ger operators with non-negative electric
  potentials in domains. The bounds reflect the correct order of
  growth in the semi-classical limit.
\end{abstract}

\author[R. L. Frank]{Rupert L. Frank}
\address{Rupert L. Frank, Department of Mathematics, Royal Institute of Technology, 100 44 Stockholm, Sweden}
\email{rupert@math.kth.se}

\author[A. Laptev]{Ari Laptev}
\address{Ari Laptev, Department of Mathematics, Imperial College London, London SW7 2AZ, UK 
$\&$ Department of Mathematics, Royal Institute of Technology, 100 44 Stockholm, Sweden}
\email{a.laptev@imperial.ac.uk $\&$ laptev@math.kth.se}

\author[S. Molchanov]{Stanislav Molchanov}
\address{Stanislav Molchanov, Department of Mathematics, University of North Carolina, Charlotte, NC 28223 - 0001, USA}
\email{smolchan@uncc.edu}

\maketitle

%%%%%%%%%%%%%%%%%%%%%%%%%%%%%%%%%%%%%%%%%%%%%%%%%%%%%%%%%%%%%%%%%%%%%%%%%%%%%%

\section{Introduction and main results}

In this paper we derive inequalities for the eigenvalues of magnetic Schr\"o\-din\-ger operators in a domain. Such inequalities have received a lot of attention in the case where both the electric and the magnetic potential are absent. If $0<\lambda_1^0<\lambda_2^0\leq \ldots$ denote the eigenvalues of the Dirichlet Laplacian in a domain $\Omega\subset\R^d$ of finite measure, then the Berezin-Li-Yau inequality \cite{B,LY} states that
\begin{equation}\label{eq:bly}
\sum_j (\lambda-\lambda^0_j)_+ \leq \frac{2}{d+2} v_d |\Omega| \lambda^{1+d/2}
\end{equation}
for any $\lambda>0$. Here
\begin{equation}\label{eq:vd}
	v_d := |\Sph^{d-1}|/d = \pi^{d/2}/\Gamma(1+d/2)
\end{equation}
denotes the volume of the $d$-dimensional unit ball. Performing a Legendre transform in \eqref{eq:bly} one obtains (see \cite{LW})
\begin{equation}\label{eq:ly}
\sum_{j=1}^k \lambda^0_j \geq \frac{4\pi^2 d}{d+2} v_d^{-2/d} |\Omega|^{-2/d} k^{1+2/d}.
\end{equation}
The important point of the estimates \eqref{eq:bly} and \eqref{eq:ly} is that they are sharp in the semi-classical limit. Indeed, according to Weyl's theorem, see, e.g., \cite{LL, RS},
\begin{equation}\label{eq:weyl}
	\lambda_k^0 \sim 4\pi^2 v_d^{-2/d} |\Omega|^{-2/d} k^{2/d},
	\qquad k\to\infty,
\end{equation}
which implies sharpness of \eqref{eq:bly} as $\lambda\to\infty$ and of \eqref{eq:ly} as $k\to\infty$.

Given the \emph{lower} bounds \eqref{eq:bly} and \eqref{eq:ly} on eigenvalue sums it is natural to ask for \emph{upper} bounds which have the correct asymptotic behavior. In \cite{L} it is shown that
\begin{equation}\label{eq:laptev0}
	\sum_j (\lambda-\lambda_j^0)_+ 
	\geq \frac{2}{(d+2)(2\pi)^d} v_d \|\omega\|_\infty^{-2} (\lambda-\lambda_1^0)_+^{1+d/2}
\end{equation}
for all $\lambda>0$ where $\omega\geq 0$ denotes the normalized
eigenfunction corresponding to the eigenvalue $\lambda_1^0$. Later it
was observed in \cite{H} that by an isoperimetric inequality of Chiti
\cite{Ch} $\|\omega\|_\infty$ can be bounded from above in terms of
the eigenvalue $\lambda_1^0$. This yields the estimate
\begin{equation}\label{eq:main0}
  \sum_j (\lambda-\lambda_j^0)_+ 
  \geq \frac2{d+2} H_d^{-1} (\lambda_1^0)^{-d/2} (\lambda -
  \lambda_1^0)_+^{1+d/2}
\end{equation}
with a non-sharp but explicit constant
\begin{equation}\label{eq:hd}
	H_d := \frac{2d}{j_{(d-2)/2}^2 J_{d/2}^2(j_{(d-2)/2})}.
\end{equation}
Here $J_\nu$ denotes the Bessel function of order $\nu$ and $j_\nu$ its first positive zero. Passing to the Legendre transform one finds
\begin{equation}\label{eq:main0legendre}
  \sum_{j=1}^k (\lambda_j^0-\lambda_1^0) 
  \leq \frac d{d+2} H_d^{2/d} \lambda_1^0 k^{1+2/d}.
\end{equation}
Both inequalities \eqref{eq:main0} and \eqref{eq:main0legendre}
reflect the correct order of growth with respect to $k$ as given by
\eqref{eq:weyl}. However, note that $\lambda_1^0$ in these
inequalities replaces $|\Omega|^{-2/d}$ in \eqref{eq:bly} and
\eqref{eq:ly}. This can be used to derive upper bounds on the
eigenvalue ratios $\lambda_k/\lambda_1$ from \eqref{eq:main0} and
\eqref{eq:main0legendre}. The sharp constant in the case $k=2$ was
found in \cite{AB}.

The purpose of the present paper is to show that the inequality \eqref{eq:main0} and its corollaries extends, with the same constant, to the case of a non-trivial magnetic field and a non-negative electric potential. More precisely, let $\Omega\subset\R^d$, $d\geq 2$, be a domain of finite measure, let $V\in L_{1,loc}(\Omega)$ be non-negative and let $A\in L_{2,loc}(\Omega,\R^d)$ be arbitrary. We define the Schr\"odinger operator
\begin{equation*}
	H := (D-A)^2 + V \qquad\text{in } L_2(\Omega)
\end{equation*}
by closing its quadratic form on $C_0^\infty(\Omega)$. We assume that this operator has compact resolvent and denote its eigenvalues by $0< \lambda_1 \leq \lambda_2 \leq \ldots$, where each eigenvalue is repeated according to multiplicities.

Our main result is

\begin{theorem}\label{main}
  For any $\lambda\geq 0$ one has
  \begin{equation}\label{eq:main}
    \sum_j (\lambda-\lambda_j)_+ 
    \geq \frac2{d+2} H_d^{-1} \lambda_1^{-d/2}(\lambda-\lambda_1)_+^{1+d/2}
  \end{equation}
  with the constant $H_d$ from \eqref{eq:hd}. Moreover, for any
  $k\in\N$ one has
  \begin{equation}\label{eq:mainlegendre}
    \sum_{j=1}^k (\lambda_j-\lambda_1) 
    \leq \frac d{d+2} H_d^{2/d} \lambda_1 k^{1+2/d}.
  \end{equation}
\end{theorem}

From these estimates on sums of eigenvalues one easily obtains
estimates on the ratio $\lambda_{k+1}/\lambda_1$.

\begin{corollary}\label{maincor}
	For any $k\in\N$ one has
	\begin{equation}\label{eq:mainsingle}
		\lambda_{k+1} \leq \left(1+ \left(1+\frac d2\right)^{2/d} H_d^{2/d} k^{2/d} \right) \lambda_1,
	\end{equation}
	\begin{equation}\label{eq:mainyang}
		\lambda_{k+1} \leq \left(1+\frac4d\right)\left(1+ \frac d{d+2} H_d^{2/d} k^{2/d} \right) \lambda_1.
	\end{equation}
\end{corollary}

We emphasize that both estimates in Corollary \ref{maincor} have the
correct order of growth as $k\to\infty$. Indeed, the asymptotics
\eqref{eq:weyl} are also valid for the eigenvalues $\lambda_k$
provided $V$, $A$ and $\Omega$ are sufficiently regular.

For large values of $k$ the bound \eqref{eq:mainyang} is better than
\eqref{eq:mainsingle}. For small values of $k$ \eqref{eq:mainsingle}
is better than \eqref{eq:mainyang}, but an even better bound is given
by
\begin{equation}\label{eq:ppw}
  \lambda_{k+1} \leq \left(1+\frac4d\right)^k \lambda_1,
  \qquad k\in\N.
\end{equation}
The latter is a Payne-P\'olya-Weinberger-type bound. The different
constants in the bounds \eqref{eq:mainsingle}, \eqref{eq:mainyang} and
\eqref{eq:ppw} are further discussed in \cite{H}.

We emphasize that all our bounds are known in the case $V\equiv A\equiv 0$, see \cite{H}. Our point is that they extend, with the same constants, to the case of a non-trivial magnetic field and a non-negative electric potential. This is non-trivial since the diamagnetic inequality apparently gives only information about the lowest eigenvalue. In particular, it seems not clear to us whether the PPW-conjecture (see \cite{AB}), namely that the quotient $\lambda_2/\lambda_1$ is maximized for a circle and $V\equiv 0$, extends to the magnetic case.

We remark that our inequalities are also true if $d=1$, but in this case much stronger results are known. (Recall that any magnetic field in one dimension can be gauged away.) In particular, if $V\geq0$, then one has the sharp inequality $\lambda_k/\lambda_1\leq k^2$, see \cite{AB2}. A beautiful proof of this inequality for $k=2$ is given in \cite{AB1}.

\subsection*{Acknowledgements}
The authors would like to thank Mark Ashbaugh for several helpful
remarks. Partial financial support through the SPECT ESF programme
(R.F., A.L.) and the Gustafsson's foundation (S.M.) is gratefully
acknowledged. The first two authors appreciate the hospitality of the
Isaac Newton Institute Cambridge where this work was completed.

%%%%%%%%%%%%%%%%%%%%%%%%%%%%%%%%%%%%%%%%%%%%%%%%%%%%%%%%%%%%%%%%%%%%%%%%%%%%%%

\section{Proof of Theorem \ref{main}}

\subsection{Main ingredients}

Our proof of Theorem \ref{main} is based on three propositions which
might be of independent interest. We defer their proofs to the
following sections. First, we derive a lower bound on the eigenvalue
sum $\sum_j (\lambda-\lambda_j)_+$ in terms of $\lambda$ and the
maximum of the lowest eigenfunction.

\begin{proposition}\label{laptev}
	Let $\omega$ be an  eigenfunction corresponding to the lowest eigenvalue $\lambda_1$ of $H$ with $\|\omega\|=1$. Then for any $\lambda\geq 0$ one has
	\begin{equation}\label{eq:laptev}
		\sum_j (\lambda-\lambda_j)_+ 
		\geq \frac{2 v_d}{(d+2)(2\pi)^d} \|\omega\|_\infty^{-2} (\lambda-\lambda_1)_+^{1+d/2}
	\end{equation}
	with $v_d$ given in \eqref{eq:vd}.
\end{proposition}

In the case $V\equiv A\equiv 0$ this is Theorem 4.1 in \cite{L}. It
turnd out that almost the same proof holds even for arbitrary $A$ and
$V\geq0$. We emphasize that in the magnetic case the lowest eigenvalue
$\lambda_1$ may be degenerate and $\|\omega\|_\infty$ may depend on
the choice of the eigenfunction. Some inequalities of the same spirit
have been obtained in \cite{S} for elliptic differential operators
with  variable coefficients.

In order to use Proposition \ref{laptev} for the proof of Theorem \ref{main} we need to estimate the maximum of the lowest eigenfunction in terms of the corresponding eigenvalue.

\begin{proposition}\label{chiti}
  Let $\lambda$ be an eigenvalue of $H$ and $\omega$ a corresponding
  eigenfunction. Then for any $p>0$
  \begin{equation}\label{eq:chiti}
    \|\omega\|_\infty \leq C_d(p) \lambda^{d/2p}\|\omega\|_p,
  \end{equation}
  with the universal constant 
  \begin{equation*}
    C_d(p) = \left( 2^{p(d-2)/2}\Gamma(d/2)^p |\Sph^{d-1}|
    \int_0^{j_{(d-2)/2}}\!\!\! J_{(d-2)/2}^p(r) r^{-p(d-2)/2+d-1}\,dr
    \right)^{-1/p}
    \!\!\!\!.
  \end{equation*}
  In particular,
  \begin{equation}\label{eq:cd2}
    C_d(2) = 
    \left( \pi^{d/2} 2^{d-2} \Gamma(d/2) j_{(d-2)/2}^2 J_{d/2}^2(j_{(d-2)/2}) \right)^{-1/2}.
  \end{equation}
\end{proposition}

In the case $V\equiv A\equiv0$ this is a result of Chiti
\cite{Ch}. Non-negative electric potentials can be included without
major changes, as was done in \cite{AB}. Using Kato's inequalitiy we
shall show that the same proof even works for a non-trivial magnetic
field.

We emphasize that \eqref{eq:chiti} is an isoperimetric result: If
$\Omega$ is a ball, $V\equiv\curl A\equiv 0$ and $\lambda$ is the
lowest eigenvalue of $H$, then one has equality in
\eqref{eq:chiti}. This explains the value of the constant $C_d(p)$. To
discuss the cases of equality would be beyond the scope of this
paper.

The third ingredient employed in the proof of our main results is the following Yang-type inequality.

\begin{proposition}\label{yang}
	For any $k\in\N$ one has
	\begin{equation}\label{eq:yang}
		\sum_{j=1}^k (\lambda_{k+1}-\lambda_j)\left(\lambda_{k+1}-\left(1+\frac4d\right)\lambda_j\right) \leq 0.
	\end{equation}
\end{proposition}

This result in the case $V\equiv A\equiv0$ is due to Yang \cite{Y}. A
slightly stronger result in the case of arbitrary $V$ and $A$ is
contained in Theorem 5 in \cite{HS}; for results in the same spirit
see \cite{LP}. A different proof (for arbitrary $V\geq 0$, but
$A\equiv 0$) was given in \cite{A}. In Section \ref{sec:yang} we shall
show that this proof extends to the case of a non-trivial magnetic
field.

One can deduce from Theorem \ref{yang} several weaker but more explicit eigenvalue estimates.

\begin{corollary}\label{yangcor}
	For any $k\in\N$ one has
	\begin{itemize}
		\item[(Y)]
		$\lambda_{k+1} \leq \left(1+\frac4d\right) \frac1k \sum_{j=1}^k\lambda_j,$
		\item[(HP)]
		$\frac 1k \sum_{j=1}^k \frac{\lambda_j}{\lambda_{k+1}-\lambda_j} \geq \frac d4,$
		\item[(PPW)]
		$\lambda_{k+1}-\lambda_k \leq \frac4d \frac1k \sum_{j=1}^k \lambda_j.$
	\end{itemize}
\end{corollary}

In the case $V\equiv A\equiv0$ these estimates are known as second Yang inequality, Hile-Protter inequality and  Payne-P\'olya-Weinberger inequality, respectively. For any $k$ one has the implications
\begin{equation}\label{eq:implic}
	\eqref{eq:yang}\Rightarrow \text{(Y)}\Rightarrow\text{(HP)}\Rightarrow\text{(PPW)}.
\end{equation}
The proof of \eqref{eq:implic} is based on general arguments (see \cite{A}) and will not be reproduced here. 

\begin{remark}\label{yangppw}
  Note that (PPW) implies \eqref{eq:ppw}. Indeed, estimating $\frac1k \sum_{j=1}^k \lambda_j \leq \lambda_k$ we obtain $\lambda_{k+1}/\lambda_k \leq 1+4/d$. Multiplying these estimates for consecutive values of $k$ we obtain \eqref{eq:ppw}.
\end{remark}

%%%%%%%%%%%%%%%%%%%%%%%%%%%%%%%%%%%%%%%%%%%%%%%%%%%%%%%%%%%%%%%%%

\subsection{Proof of the main results}

We now show how Theorem \ref{main} and Corollary \ref{maincor} follow
from Propositions \ref{laptev}, \ref{chiti} and \ref{yang}.

\begin{proof}[Proof of Theorem \ref{main}]
  To get \eqref{eq:main} we simply combine the estimates
  \eqref{eq:laptev} and \eqref{eq:chiti} and note that $H_d = (2\pi)^d
  v_d^{-1} C_d(2)^2$.

  To deduce \eqref{eq:mainlegendre} from \eqref{eq:main} we use the
  Legendre transform argument from \cite{LW}. Indeed, for a function
  $h$ on $[0,\infty)$ put $Lh(p) :=
  \sup_{\lambda>0}(p\lambda-h(\lambda))$. So if $f(\lambda):=\sum_j
  (\lambda-\lambda_j)_+$ and $g(\lambda):= \frac2{d+2} H_d^{-1}
  \lambda_1^{-d/2}(\lambda_1-\lambda)_-^{1+d/2}$, then
  \begin{equation*}
    Lf(p) = \{p\}\lambda_{[p]+1} + \sum_{j=1}^{[p]} \lambda_j
  \end{equation*}
  (with $\{p\}$, $[p]$ the fractional and integer parts of $p$) and
  \begin{equation*}
    Lg(p) = p\lambda_1 + \frac d{d+2} H_d^{2/d} \lambda_1 p^{1+2/d}.
  \end{equation*}
  Now $f\geq g$ implies $Lf\leq Lg$, and evaluating the latter at an
  integer $p=k\in\N$ we obtain \eqref{eq:mainlegendre}.
\end{proof}

\begin{proof}[Proof of Corollary \ref{maincor}]
  For the proof of \eqref{eq:mainsingle} we put
  $\lambda=\lambda_{k+1}$ in \eqref{eq:main} and estimate
  $\sum_{j=1}^k(\lambda_{k+1}-\lambda_j) \leq
  k(\lambda_{k+1}-\lambda_1)$. For the proof of \eqref{eq:mainsingle}
  we combine (Y) with \eqref{eq:mainlegendre}.
\end{proof}

%%%%%%%%%%%%%%%%%%%%%%%%%%%%%%%%%%%%%%%%%%%%%%%%%%%%%%%%%%%%

\section{Upper bound on eigenvalue sums}

In this section we prove Proposition \ref{laptev}. We let $u_1:=\omega$ and denote by $u_j$, $j\in\N$, orthonormal eigenfunctions corresponding to the eigenvalues $\lambda_j$. We put $\theta_\xi(x) := e^{i\xi\cdot x}\omega(x)$, $\tilde\omega:=\|\omega\|_\infty$ and calculate as in \cite{L} using Plancherel's equality
	\begin{equation}\label{eq:laptev1}
		\begin{split}
			\sum_j (\lambda-\lambda_j)_+
			& \geq \tilde\omega^{-2} \sum_j (\lambda-\lambda_j)_+ \int_\Omega |\overline\omega u_j|^2\,dx \\
			& = (2\pi)^{-d} \tilde\omega^{-2} \sum_j (\lambda-\lambda_j)_+ 
			\int_{\R^d} \left| \int_\Omega \overline\omega(x) u_j(x) e^{-i\xi\cdot x} \,dx \right|^2 \,d\xi \\
			& = (2\pi)^{-d} \tilde\omega^{-2}
			\int_{\R^d} \sum_j (\lambda-\lambda_j)_+  \left| (\theta_\xi, u_j) \right|^2\,d\xi \\
			& = (2\pi)^{-d} \tilde\omega^{-2}
			\int_{\R^d} \left( \int_0^\infty (\lambda-\nu)_+  d(E(\nu)\theta_\xi,\theta_\xi) \right) \,d\xi.
		\end{split}
	\end{equation}
	Here $d E(\nu)$ denotes the spectral measure of $H$. We note that
	\begin{equation*}
		\int_0^\infty d(E(\nu)\theta_\xi,\theta_\xi) = \|\theta_\xi\|^2 = \|\omega\|^2 = 1.
	\end{equation*}
	Since the function $\nu\mapsto (\lambda-\nu)_+$ is convex, we can apply Jensen's inequality to get
	\begin{equation}\label{eq:jensen}
		\begin{split}
			\int_0^\infty (\lambda-\nu)_+  d(E(\nu)\theta_\xi,\theta_\xi) 
			& \geq \left(\lambda - \int_0^\infty \nu d(E(\nu)\theta_\xi,\theta_\xi) \right)_+ \\
			& = \left(\lambda - \int_\Omega \left( |(D-A)\theta_\xi|^2 + V|\theta_\xi|^2 \right)\,dx \right)_+.
		\end{split}
	\end{equation}
	We shall need the following simple
	
\begin{lemma}
	For any $\xi\in\R^d$ one has
	\begin{equation}\label{eq:square}
		\int_\Omega \left( |(D-A)\theta_\xi|^2 + V|\theta_\xi|^2 \right)\,dx
		= \lambda_1 + |\xi|^2.
	\end{equation}
\end{lemma}

\begin{proof}
	We denote by $h$ the quadratic form of the operator $H$, i.e.,
	\begin{equation*}
		h[u] = \int_\Omega \left( |(D-A)u|^2 + V|u|^2 \right)\,dx.
	\end{equation*}	
	An easy calculation shows that
	\begin{equation}\label{eq:square1}
		h[\theta_\xi] = \lambda_1 + |\xi|^2 + 2 j\cdot\xi 
	\end{equation}
	where
	\begin{equation*}
		j := \im \int_\Omega \overline\omega (\nabla-iA)\omega \,dx,
	\end{equation*}
	and we have to show that $j=0$. Assume to the contrary that $j\neq 0$ and choose $\xi_0$ such that $j\cdot\xi_0<0$. Then in view of \eqref{eq:square1}, $h[\theta_{\epsilon\xi_0}]<\lambda_1$ for all sufficiently small $\epsilon>0$. But this contradicts the variational characterization of the lowest eigenvalue.
\end{proof}

	Combining \eqref{eq:laptev1}, \eqref{eq:jensen} and \eqref{eq:square} we arrive at
	\begin{equation*}
		\begin{split}
			\sum_j (\lambda-\lambda_j)_+
			& \geq (2\pi)^{-d} \tilde\omega^{-2}
			\int_{\R^d} \left(\lambda-\lambda_1 - |\xi|^2\right)_+ \,d\xi \\
			& = (2\pi)^{-d} \tilde\omega^{-2} v_d \frac 2{d+2} \left(\lambda-\lambda_1\right)_+^{1+d/2},
		\end{split}
	\end{equation*}
	which proves the assertion.

%%%%%%%%%%%%%%%%%%%%%%%%%%%%%%%%%%%%%%%%%%%%%%%%%%%%%%%%%%%%%%%%%%%%%%%%%%%%%%%%%%%%%%%%%%%%%%%%%%%%%%%%%%%%%%

\section{Eigenfunction estimates for magnetic Schr\"odinger operators}

\subsection{A non-sharp bound}

As we pointed out after Proposition \ref{chiti}, the constants in \eqref{eq:chiti} are sharp. Before giving its rather involved proof we would like to derive a (non-sharp) estimate which still has the correct form needed in the proof of Theorem \ref{main}. More precisely, we shall prove that if $\omega$ is an eigenfunction corresponding to an eigenvalue $\lambda$ of $H$, then
\begin{equation*}
	\|\omega\|_\infty \leq \tilde C_d \lambda^{d/4}\|\omega\|_2,
\end{equation*}
with
\begin{equation*}
	\tilde C_d = (e/d\pi)^{d/4}.
\end{equation*}
	
To see this, we note that by the diamagnetic inequality and the Feynman-Kac formula \cite{D} one has for all $t>0$ and all $x\in\Omega$ the pointwise estimate
\begin{equation*}
	\left| e^{-t\lambda}\omega(x) \right| 
	= \left|\left(e^{-tH}\omega\right)(x)\right|
	\leq \left(e^{t\Delta}|\omega|\right)(x).
\end{equation*}
Here $|\omega|$ is extended by $0$ to $\R^d$ and $e^{t\Delta}$ denotes the heat semi-group on $\R^d$. Using the explicit expression of its kernel we can estimate
\begin{equation*}
	\left(e^{t\Delta}|\omega|\right)(x)
	\leq (4\pi t)^{-d/2} \left( \int_{\R^d} e^{-|x-y|^2/2t}\,dy \right)^{1/2} \|\omega\|_2
	= (4\pi t)^{-d/4} \|\omega\|_2.
\end{equation*}
Combining the previous two relations and optimizing with respect to $t$ we arrive at
\begin{equation*}
	\left|\omega(x) \right| 
	\leq \inf_{t>0} e^{t\lambda} (4\pi t)^{-d/4} \|\omega\|_2
	=  \left(e/d\pi\right)^{d/4} \lambda^{d/4} \|\omega\|_2,
\end{equation*}
as claimed.
	
%%%%%%%%%%%%%%%%%%%%%%%%%%%%%%%%%%%%%%%%%%%%%%%%%%%%%%%%%%%%%%%%%%%%%%%%

\subsection{Proof of Proposition \ref{chiti}}

Now we turn to the sharp bound \eqref{eq:chiti}. We introduce the ball
\begin{equation*}
	S:= \{ x\in\R^d : |x| < j_{(d-2)/2}\lambda^{-1/2} \}
\end{equation*}
and the function
\begin{equation*}
	z(x) := |x|^{-(d-2)/2} J_{(d-2)/2}(\sqrt\lambda |x|).
\end{equation*}
Note that $S$ is the ball centered at the origin for which the Dirichlet problem (with $V\equiv A\equiv 0$) has lowest eigenvalue $\lambda$. The function $z$ is the corresponding eigenfunction. Next we introduce the ball
\begin{equation}\label{eq:rearrangeomega}
	\Omega^* := \{ x\in\R^d : |x| <  v_d^{-1/d} |\Omega|^{1/d} \}
\end{equation}
centered at the origin with the same volume as $\Omega$. The constant $v_d$ is the volume of the $d$-dimensional unit ball, see \eqref{eq:vd}.

\begin{lemma}\label{faberkrahn}
	One has $S\subset\Omega^*$.
	\iffalse
	with equality iff $\Omega$ is a ball, $V\equiv\curl A\equiv 0$ and $\lambda$ is the lowest eigenvalue of $H$.
	\fi
\end{lemma}

\begin{proof}	
	We denote by $\lambda_1(-\Delta,\mathcal O)$ the lowest eigenvalue of the Dirichlet Laplacian on a domain $\mathcal O$. Note that by definition of $S$
	\begin{equation*}
		\lambda_1(-\Delta,S)=\lambda.
	\end{equation*}
	Now the variational principle together with the diamagnetic inequality and the non-negativity of $V$ imply
	\begin{equation*}
		\lambda \geq \lambda_1(-\Delta,\Omega).
	\end{equation*}
	Finally, the Faber-Krahn inquality states that
	\begin{equation*}
		\lambda_1(-\Delta,\Omega) \geq \lambda_1(-\Delta,\Omega^*).
	\end{equation*}
	\iffalse
	with equality iff $\Omega$ is a ball. 
	\fi
	Combining the previous relations we obtain $\lambda_1(-\Delta,S)\geq\lambda_1(-\Delta,\Omega^*)$. By the variational principle, this implies $S\subset\Omega^*$.
	\iffalse 
	As already mentioned, if equality holds, then $\Omega$ is a ball. Moreover, the lowest eigenvalues of $-\Delta$ and $(D-A)^2+V$ on $\Omega$ coincide. This implies $V\equiv 0$ (since the lowest Laplace eigenfunction is non-zero everywhere) and $\curl A\equiv0$.\marginpar{Check! Ball simply-connected}
	\fi
\end{proof}

We denote by $\omega^*:\Omega^*\rightarrow[0,\infty)$ the spherically decreasing rearrangement of $\omega$, see \cite{LL}. We note that $z$ is a spherically decreasing function. (This follows either by properties of Bessel functions or by rearrangement arguments.) In particular,
\begin{equation}\label{eq:zorigin}
	\|z\|_\infty = z(0).
\end{equation}
The core of Proposition \ref{chiti} is the following comparison result.

\begin{lemma}\label{comparison}
	Assume that $\omega$ is normalized such that $\|\omega\|_{L_\infty(\Omega)} = \|z\|_{L_\infty(S)}$. Then
	\begin{equation*}
		\omega^*(x) \geq z(x), \qquad x\in S.
	\end{equation*}
\end{lemma}

We defer the rather involved proof of this lemma to the next subsection. Now we shall use it to give the

\begin{proof}[Proof of Proposition \ref{chiti}]
	By Lemma \ref{comparison} one has
	\begin{equation*}
		\|z\|_{L_\infty(S)} \omega^*(x) \geq \|\omega\|_{L_\infty(\Omega)} z(x),
		\qquad x\in S.
	\end{equation*}
	Hence according to rearrangement properties \cite{LL} and Lemma \ref{faberkrahn}
	\begin{equation}\label{eq:mainproof}
		\|\omega \|_{L_p(\Omega)} = \|\omega^* \|_{L_p(\Omega^*)} \geq \|\omega^* \|_{L_p(S)} 
		\geq \frac{\|z\|_{L_p(S)}}{\|z\|_{L_\infty(S)}} \|\omega\|_{L_\infty(\Omega)}.
	\end{equation}
	The asymptotics $r^{-(d-2)/2}J_{(d-2)/2}(r)\to \Gamma(d/2)^{-1} 2^{-(d-2)/2}$ as $r\to 0$ (see \cite[(9.1.7)]{AS}) and \eqref{eq:zorigin} imply
	\begin{equation*}
		\frac{\|z\|_{L_\infty(S)}}{\|z\|_{L_p(S)}} = C_d(p) \lambda^{d/2p}.
	\end{equation*}
	In the special case $p=2$ we use (see \cite[(11.4.5)]{AS})
	\begin{equation*}
		\int_0^{j_{(d-2)/2}} r J_{(d-2)/2}^2(r) \,dr = \frac 12 j_{(d-2)/2}^2 J_{d/2}^2(j_{(d-2)/2})
	\end{equation*}
	and $|\Sph^{d-1}|=d\pi^{d/2}/\Gamma((d+2)/2)$ to get the claimed expression.
\end{proof}

\subsection{Proof of Lemma \ref{comparison}}

Since $z$ is a radial function, there is a function $v$ on $[0,|S|]$ such that
\begin{equation*}
	v(v_d |x|^d) = z(x), \qquad x\in S.
\end{equation*}
Here $v_d$ is the volume of the unit ball, see \eqref{eq:rearrangeomega}. From the differential equation satisfied by $z$ one derives easily the integro-differential equation
\begin{equation}\label{eq:integraleq}
	-v'(s) = d^{-2} v_d^{-2/d} \lambda s^{-2+2/d} \int_0^s v(t)\,dt,
	\qquad 0< s < |S|,
\end{equation}
for $v$. Our next goal is to derive a similar integro-differential \emph{inequality} for $\omega^*$. Namely, we let $u$ be the function on $[0,|\Omega|]$ such that
\begin{equation*}
	u(v_d |x|^d) = \omega^*(x), \qquad x\in \Omega^*.
\end{equation*}
Note that $u$ is a decreasing function with $u(0)=\|\omega\|_\infty$ and $u(|\Omega|)=0$.

\begin{lemma}\label{integralineq}
	One has
	\begin{equation}\label{eq:integralineq}
		-u'(s) \leq d^{-2} v_d^{-2/d} \lambda s^{-2+2/d} \int_0^s u(t)\,dt,
		\qquad 0< s < |\Omega|.
	\end{equation}
\end{lemma}

\begin{proof}
	We recall Kato's inequality \cite{S1, S2},
	\begin{equation*}
		\re\left( (\sgn w) (D-A)^2 w\right) \geq -\Delta |w|,
	\end{equation*}
	where $\sgn w(x) = \overline{w(x)}/|w(x)|$ if $w(x)\neq 0$ and $\sgn w(x)=0$ otherwise.	Since $V\geq 0$ we have
	\begin{equation*}
		\lambda |\omega| = \re\left( (\sgn\omega) ((D-A)^2+V)\omega\right) \geq -\Delta |\omega|.
	\end{equation*}
	We integrate this inequality over the set $\{|\omega|>t\}$ and use Gauss's theorem in the form
	\begin{equation*}
		-\int_{\{|w|>t\}} \Delta |w| \,dx = \int_{\{|w|=t\}} |\nabla |w|| \,d\sigma.
	\end{equation*}
	Here $d\sigma$ denotes $(d-1)$-dimensional surface measure. We obtain
	\begin{equation}\label{eq:intineq1}
		\lambda \int_{\{|\omega|>t\}} |\omega| \,dx \geq \int_{\{|\omega|=t\}} |\nabla |\omega|| \,d\sigma.
	\end{equation}
	
	The next step is to estimate the surface integral on the RHS from below. We introduce the distribution function
	\begin{equation*}
		\mu(t) := |\{x\in\Omega : |\omega|>t \}|
	\end{equation*}
	and recall the coarea formula
	\begin{equation*}
		-\mu'(t) = \int_{\{|\omega|=t\}} |\nabla |\omega||^{-1} \,d\sigma.
	\end{equation*}
	Hence by Cauchy-Schwarz
	\begin{equation*}
		\begin{split}
			\sigma(\{|\omega|=t\})^2 
			& \leq \left(\int_{\{|\omega|=t\}} |\nabla |\omega||^{-1} \,d\sigma\right)
			\left(\int_{\{|\omega|=t\}} |\nabla |\omega|| \,d\sigma\right)\\
			& = -\mu'(t) \int_{\{|\omega|=t\}} |\nabla |\omega|| \,d\sigma.
		\end{split}
	\end{equation*}
	Combining this with the isoperimetric inequality
	\begin{equation*}
		\sigma(\{|\omega|=t\}) \geq d v_d^{1/d} \mu(t)^{1-1/d}
	\end{equation*}
	we finally arrive at
	\begin{equation*}\label{eq:intineq2}
		\int_{\{|\omega|=t\}} |\nabla |\omega|| \,d\sigma	\geq d^2 v_d^{2/d} \mu(t)^{2-2/d} (-\mu'(t))^{-1}.
	\end{equation*}
	
	Combining this with \eqref{eq:intineq1} we have shown that
	\begin{equation*}
		\mu(t)^{2-2/d} (-\mu'(t))^{-1} \leq d^{-2} v_d^{-2/d} \lambda \int_{\{|\omega|>t\}} |\omega| \,dx.
	\end{equation*}
	It remains to substitute $t=u(s)$ and to note that when $u(s)\neq 0$,
	\begin{equation*}
		\mu(t)^{2-2/d} (-\mu'(t))^{-1} = s^{2-2/d} (-u'(s)),
		\qquad \int_{\{|\omega|>t\}} |\omega| \,dx = \int_0^s u(r)\,dr.
	\end{equation*}
	This proves the assertion.
\end{proof}

Now everything is in place for the

\begin{proof}[Proof of Lemma \ref{comparison}]
	The strategy is to construct a trial function $g$ for $-\Delta$ on $S$ which is as good as $z$ and hence must coincide with $z$. We distinguish two cases, namely whether $|S|<|\supp\omega|$ or not. We begin by assuming that this strict inequality holds. Recall the definition of the functions $u$, $v$ before Lemma \ref{integralineq}. Both functions are decreasing and the normalization condition $\|\omega\|_{L_\infty(\Omega)} = \|z\|_{L_\infty(S)}$ means $u(0)=v(0)$. Moreover, $u$ is strictly positive on $[0,|S|]$. (This is where we use that $|S|<|\supp\omega|$.) Hence
	\begin{equation*}
		c := \max_{0\leq s\leq |S|} v(s)/u(s)
	\end{equation*}
	is finite and greater or equal to $1$. We have to prove that $c=1$. For this we choose a point $s_0\in [0,|S|)$ such that $v(s_0)=c u(s_0)$. (The point can be chosen smaller than $|S|$ since $v(|S|)=0$.) Defining
	\begin{equation*} 
		w(s) := \left\{ 
		\begin{array}{ll}
			c u(s), & 0 \leq s \leq s_0, \\
			v(s), & s_0 \leq s \leq |S|,
		\end{array} \right. 
	\end{equation*}
	it follows from \eqref{eq:integraleq}, \eqref{eq:integralineq} and the definition of $c$ that $w$ satisfies
	\begin{equation}\label{eq:integralineqw}
		-w'(s) \leq d^{-2} v_d^{-2/d} \lambda s^{-2+2/d} \int_0^s w(t)\,dt,
		\qquad 0< s < |S|.
	\end{equation}
	Now we define $g(x):=w(v_d|x|^d)$ for $x\in S$ and note that $g\in \overset{\circ}{H^1}(S)$ with
	\begin{equation*}
		\int_S |g|^2 \,dx = \int_0^{|S|} |w|^2\,ds
	\end{equation*}
	and, using \eqref{eq:integralineqw},
	\begin{equation*}
		\begin{split}
			\int_S |\nabla g|^2 \,dx & = d^2 v_d^{2/d} \int_0^{|S|} |w'|^2 s^{2-2/d}\,ds \\
			& \leq \lambda \int_0^{|S|} (-w'(s)) \int_0^s w(r)\,dr\,ds \\
			& = \lambda \int_0^{|S|} |w|^2\,ds.
		\end{split}
	\end{equation*}
	Recall that $\lambda$ is the lowest eigenvalue of $-\Delta$ on $S$. Since it is simple and has $z$ as a corresponding eigenfunction, we conclude that $g$ is a scalar multiple of $z$. Hence $w$ is a scalar multiple of $v$. Since $w$ and $v$ coincide on $[s_0,|S|]$ they coincide everywhere. Hence $c u(s) = v(s)$ for all $0\leq s\leq s_0$. Evaluating at $s=0$ we find $c=1$, which proves the assertion in the case under consideration.

	Now we turn to the case $|S|\geq |\supp\omega|$. Note that by Lemma \ref{faberkrahn} we know that $S\subset\Omega^*$. We extend $\omega^*$ by $0$ to $S$ (if $|\supp\omega|<|S|$), and denote the resulting function by $g$. Lemma \ref{integralineq} and the same calculation as in the first part of the proof imply that $g\in \overset{\circ}{H^1}(S)$ with
	\begin{equation*}
		\int_S |\nabla g|^2 \,dx \leq \lambda \int_S |g|^2\,dx.
	\end{equation*}
	As before, this implies that $g$ is a scalar multiple of $z$ and, after evaluation at $x=0$, that $g=z$. This proves the assertion also in this case.
\end{proof}

%%%%%%%%%%%%%%%%%%%%%%%%%%%%%%%%%%%%%%%%%%%%%%%%%%%%%%%%%%%%%%%%%%%%%%%%%%%%%%%%%%%%%%%%%%%%%%%%%%

\section{Yang inequalities for magnetic Schr\"odinger operators}\label{sec:yang}
 
We turn now to the proof of Proposition \ref{yang}. 
We work under the same assumptions on $\Omega$, $V$ and $A$ as in the previous sections.
Again we denote by $u_i$, $i\in\N$, orthonormal eigenfunctions corresponding to the eigenvalues $\lambda_i$. Throughout the proof we will fix $k\in\N$. Moreover, for any $1\leq i\leq k$ and any $1\leq l\leq d$ we define the function
\begin{equation*}
	\phi_i := x_l u_i - \sum_{j=1}^k a_{ij}u_j,
	\qquad a_{ij} := \int x_l u_i \overline{u_j}\,dx,
\end{equation*}
and note that
\begin{equation}\label{eq:ortho}
	(\phi_i,u_j) = 0, \qquad 1\leq j\leq k.
\end{equation}
Hence the variational principle implies
\begin{equation}\label{eq:varprin}
	\lambda_{k+1} \leq 
	\frac{\int\left( |(D-A)\phi_i|^2+V|\phi_i|^2\right)\,dx}{\int |\phi_i|^2\,dx} 
\end{equation}
(with the convention that the RHS is infinite if $\phi_i\equiv 0$). Using the commutator identity
\begin{equation}\label{eq:commut}
	[(D-A)^2,x_l] = -2\imunit (D_l-A_l)
\end{equation}
and the orthogonality \eqref{eq:ortho} one easily finds
\begin{equation*}
	\int\left( |(D-A)\phi_i|^2+V|\phi_i|^2\right)\,dx 
	= \lambda_i \int |\phi_i|^2\,dx + 2\im\int (D_l-A_l)u_i \overline{\phi_i} \,dx,
\end{equation*}
which, together with \eqref{eq:varprin}, leads to
\begin{equation}\label{eq:evdiff}
	\lambda_{k+1} - \lambda_i \leq \frac{2\im\int (D_l-A_l)u_i \overline{\phi_i} \,dx}{\int |\phi_i|^2\,dx}.
\end{equation}

Next, we manipulate the numerator. We note that
\begin{equation*}
	\begin{split}
		2\im\int (D_l-A_l)u_i \overline{x_l u_i} \,dx
		& = - 2\re\int \frac{\partial u_i}{\partial x_l} \overline{x_l u_i}\,dx
		= -\int x_l \frac{\partial |u_i|^2}{\partial x_l}\,dx\\
		& = \int |u_i|^2\,dx 
		= 1,
	\end{split}
\end{equation*}
and hence
\begin{equation}\label{eq:numeq}
	2\im\int (D_l-A_l)u_i \overline{\phi_i} \,dx
	= 1 + 2\re\sum_{j=1}^k \overline{a_{ij}} b_{ij}
\end{equation}
where we have set
\begin{equation*}
	b_{ij} := \imunit \int (D_l-A_l)u_i \overline{u_j}\,dx.
\end{equation*}
With the help of \eqref{eq:commut} we calculate
\begin{equation*}
	2 b_{ij} = - \int \left( x_l u_i \overline{(D-A)^2 u_j}- (D-A)^2u_i \overline{x_l u_j}\right) \,dx
	= (\lambda_i-\lambda_j) a_{ij}.
\end{equation*}
Hence \eqref{eq:numeq} becomes
\begin{equation}\label{eq:numeq2}
	2\im\int (D_l-A_l)u_i \overline{\phi_i} \,dx
	= 1 + \sum_{j=1}^k (\lambda_i-\lambda_j)|a_{ij}|^2
\end{equation}
Now we use the Cauchy-Schwarz inequality to estimate the numerator in \eqref{eq:evdiff}. Note that by \eqref{eq:ortho}
\begin{equation*}
  \im\int (D_l-A_l)u_i \overline{\phi_i} \,dx 
  = \im\int \left((D_l-A_l)u_i
  +\imunit\sum_{j=1}^k b_{ij}u_j\right) \overline{\phi_i} \,dx
\end{equation*}
and hence
\begin{equation*}
  \begin{split}
    & \left(\im\int (D_l-A_l)u_i \overline{\phi_i} \,dx\right)^2 \\
    & \qquad \leq \left(\int |\phi_i|^2\,dx\right)
    \left(\int \left|(D_l-A_l)u_i
    +\imunit\sum_{j=1}^k b_{ij}u_j\right|^2\,dx\right) \\
    & \qquad = \left(\int |\phi_i|^2\,dx\right) 
    \left(\int |(D_l-A_l)u_i|^2\,dx - \sum_{j=1}^k |b_{ij}|^2 \right)\\
    & \qquad = \left(\int |\phi_i|^2\,dx\right) 
    \left(\int |(D_l-A_l)u_i|^2\,dx 
    - \frac14\sum_{j=1}^k (\lambda_i-\lambda_j)^2|a_{ij}|^2 \right).
  \end{split}
\end{equation*}
Combining this inequality with \eqref{eq:evdiff} and \eqref{eq:numeq2} we arrive at
\begin{equation*}
	\lambda_{k+1} - \lambda_i 
	\leq \frac{4\int |(D_l-A_l)u_i|^2\,dx - \sum_{j=1}^k (\lambda_i-\lambda_j)^2|a_{ij}|^2 }
	{1 + \sum_{j=1}^k (\lambda_i-\lambda_j)|a_{ij}|^2}.
\end{equation*}
Equivalently,
\begin{equation}\label{eq:evdiff2}
	(\lambda_{k+1} - \lambda_i) + (\lambda_{k+1}-\lambda_j)(\lambda_i-\lambda_j)\sum_{j=1}^k |a_{ij}|^2	
	\leq 4\int |(D_l-A_l)u_i|^2\,dx.
\end{equation}
Now we write $a_{ijl}$ instead of $a_{ij}$ to emphasize the dependence on $l$, and define
\begin{equation*}
	A_{ij} := \sum_{l=1}^d |a_{ijl}|^2.
\end{equation*}
Summing \eqref{eq:evdiff2} over $l$ and using that $V\geq 0$ we obtain
\begin{equation*}
	d(\lambda_{k+1} - \lambda_i) + (\lambda_{k+1}-\lambda_j)(\lambda_i-\lambda_j)\sum_{j=1}^k A_{ij}	
	\leq 4\lambda_i.
\end{equation*}
Note that $A_{ij}=A_{ji}$. In order to (anti-)symmetrize, we multiply by $(\lambda_{k+1}-\lambda_i)$ and sum over $i$. The resulting inequality is
\begin{equation*}
	d\sum_{i=1}^k(\lambda_{k+1} - \lambda_i)^2 \leq 4\sum_{i=1}^k \lambda_i(\lambda_{k+1} - \lambda_i).
\end{equation*}
This completes the proof of Proposition \ref{yang}.

%%%%%%%%%%%%%%%%%%%%%%%%%%%%%%%%%%%%%%%%%%%%%%%%%%%%%%%%%%%%%%%%%%%%%%%%%%%%%%%%%%%%%%%%%%%%%%%%%%%%%%%%%%%%%%%%%%%%%

\bibliographystyle{amsalpha}

\end{document}